\documentclass[a4paper,11pt]{amsart}

\title[Approximate isomorphism of randomization pairs]{Approximate isomorphism of randomization pairs}

\author[J. Hanson]{James Hanson}
\address{James Hanson
  \\ Department of Mathematics \\
  University of Maryland \\
  4176 Campus Dr. \\
  College Park, MD 20742-4015 \\
  USA.}

\author[T. Ibarluc\'ia]{Tom\'as Ibarluc\'ia}
\thanks{Research partially supported by the ANR contract AGRUME (ANR-17-CE40-0026).}
\address{Tom\'as Ibarluc\'ia
\\ Universit\'e Paris Cit\'e \\
  CNRS \\
  IMJ-PRG \\
  F-75006 Paris \\
  France.}
\urladdr{https://webusers.imj-prg.fr/~tomas.ibarlucia}

\usepackage{amsmath, amssymb, amsthm, amscd, mathtools, enumitem, color}

\usepackage[T1]{fontenc}
\usepackage{lmodern}

\usepackage{setspace}
\usepackage{geometry}
\usepackage{stmaryrd}

\makeatletter
\g@addto@macro\bfseries{\boldmath}
\makeatother

\def\mbN{\mathbb{N}}
\def\mbZ{\mathbb{Z}}

\def\mbP{\mathbb{P}}

\def\rd{\mathrm{d}}

\def\mcB{\mathcal{B}}

\def\mcL{\mathcal{L}}

\def\mcT{\mathcal{T}}
\def\mcD{\mathcal{D}}

\DeclareMathOperator{\Aut}{Aut}
\DeclareMathOperator{\End}{End}

\DeclareMathOperator{\Sym}{Sym}
\DeclareMathOperator{\acl}{acl}

\newcommand{\eq}{\mathrm{eq}}
\newcommand{\inv}{^{-1}}
\newcommand{\ov}[1]{\overline{#1}}
\newcommand{\actson}{\curvearrowright}

\theoremstyle{plain}        \newtheorem{fact}{Fact}[section]
\theoremstyle{plain}        \newtheorem{theorem}[fact]{Theorem}
\theoremstyle{plain}        \newtheorem{lem}[fact]{Lemma}
\theoremstyle{plain}        \newtheorem{prop}[fact]{Proposition}
\theoremstyle{plain}        
\theoremstyle{definition}   \newtheorem{rem}[fact]{Remark} 
\theoremstyle{definition}   \newtheorem{defin}[fact]{Definition}
\theoremstyle{definition}   \newtheorem*{convention}{Convention}
\theoremstyle{definition}   
\theoremstyle{definition}   
\theoremstyle{definition}   
\theoremstyle{definition}   \newtheorem{question}[fact]{Question}
\theoremstyle{definition}   
\theoremstyle{plain}        \newtheorem*{conjecture}{Conjecture}
\theoremstyle{plain}        \newtheorem*{theorem*}{Theorem}
\theoremstyle{plain}        \newtheorem*{claim*}{Claim}

\usepackage[hidelinks]{hyperref}

\begin{document}

\begin{abstract}
We study approximate $\aleph_0$-categoricity of theories of beautiful pairs of randomizations, in the sense of continuous logic.

This leads us to disprove a conjecture of Ben Yaacov, Berenstein and Henson, by exhibiting $\aleph_0$-categorical, $\aleph_0$-stable metric theories $Q$ for which the corresponding theory $Q_P$ of beautiful pairs is not approximately $\aleph_0$-categorical, i.e., has separable models that are not isomorphic even up to small perturbations of the smaller model of the pair. The theory $Q$ of randomized infinite vector spaces over a finite field is such an example.

On the positive side, we show that the theory of beautiful pairs of randomized infinite sets is approximately $\aleph_0$-categorical. We also prove that a related stronger property, which holds in that case, is preserved under various natural constructions, and formulate our guesswork for the general case.
\end{abstract}

\maketitle

\section*{Introduction}

Throughout the paper we let $\mcT$ be the class of $\aleph_0$-categorical, $\aleph_0$-stable, classical (i.e., discrete) first-order theories in a countable language. This class is considered to be particularly well-behaved and well-understood in many respects, thanks to the achievements of the geometric stability theory program from the eighties. Deep results of Cherlin, Harrington, Lachlan, Mills and Zil'ber established, for instance, that every $T\in\mcT$ is one-based and has finite Morley rank, and that the geometry of any strictly minimal set interpretable in $T$ is either disintegrated or that of a projective or affine space over a finite field.

The geometric theory, as it is, breaks down completely in the setting of continuous logic, where appropriate analogs of most of the notions involved are lacking. In \cite{bbhSFB14}, Ben Yaccov, Berenstein and Henson introduced a new notion, \emph{strongly finitely based theories}, containing the class $\mcT$ as well as the metric theories of Hilbert spaces and of probability algebras. We need not recall the original definition, but the following equivalence: a stable $\aleph_0$-categorical metric theory $Q$ is strongly finitely based if and only if the theory $Q_P$ of \emph{beautiful pairs} of models of $Q$ is $\aleph_0$-categorical. They then argued that this property could be an appropriate metric generalization of one-basedness within the class of stable, separably categorical theories. Or almost so.

Indeed, as they pointed out, not every $\aleph_0$-categorical, $\aleph_0$-stable metric theory is strongly finitely based. A counterexample is given by the theory of atomless $L^p$ Banach lattices, for which the corresponding theory of beautiful pairs admits exactly two non-isomorphic separable models. On the other hand, these two models are \emph{approximately isomorphic}, in a natural general sense that we recall in \textsection \ref{ss:beautiful-pairs}. They thus proposed the following:

\begin{conjecture}[{\cite[4.19]{bbhSFB14}}\footnote{Note the shift in the numbering of sections and results in the published version of \cite{bbhSFB14} and the available preprint versions.}]
If $Q$ is an $\aleph_0$-categorical, $\aleph_0$-stable theory in a countable language, then $Q_P$ is approximately $\aleph_0$-categorical.
\end{conjecture}

In \cite{ibaRando}, the second author identified another source of examples of separably categorical, $\aleph_0$-stable theories $Q$ for which $Q_P$ admits several non-isomorphic separable models. It suffices to consider any $T\in\mcT$ with infinite models and let $Q=T^R$ be its \emph{randomization}, in the sense of Ben Yaacov--Keisler \cite{benkei}. We recall this result in \textsection \ref{ss:randomizations} below. The present work initiated as an effort to prove the above conjecture in the case of the theories $Q=T^R$, $T\in\mcT$. With some surprise, it was soon apparent that the conjecture does not always hold in this class.

Let $\mcD\subseteq\mcT$ be the subclass of those $T$ such that every strictly minimal set interpretable in $T$ has a disintegrated geometry. In Section~\ref{s:counterexamples} we show the following.

\begin{theorem*}
If $T\in \mcT\setminus\mcD$, then $(T^R)_P$ is not approximately $\aleph_0$-categorical.
\end{theorem*}

For a concrete example, let $V$ be a countably infinite vector space over a finite field, and let $W\subseteq V$ be an infinite subspace with infinite codimension. Let $V^{\Omega^2}\coloneqq L([0,1]^2,V)$ denote the randomized vector space consisting of $V$-valued random variables on the unit square. Let also $W^\Omega$ (respectively, $V^\Omega$) denote the substructure of $V^{\Omega^2}$ consisting of those random variables that depend only on the first coordinate of the unit square and take values in $W$ (resp., in $V$). Then $(V^{\Omega^2},W^\Omega)$ is a separable beautiful pair of randomized vector spaces, and $(V^{\Omega^2},V^\Omega)$ is a separable pair with the same theory, but there is no $\epsilon$-isomorphism between them for any $\epsilon<1$.

On the other hand, if we remove the vector space structure from $V$ and see it just as an infinite set, and see $W$ as an infinite, coinfinite subset, we are able to construct $\epsilon$-isomorphisms between the pairs $(V^{\Omega^2},W^\Omega)$ and $(V^{\Omega^2},V^\Omega)$ for every $\epsilon>0$. In fact, in Section~\ref{s:positive-results} we prove:

\begin{theorem*}
If $T$ is the theory of infinite sets, then $(T^R)_P$ is approximately $\aleph_0$-categorical.
\end{theorem*}

To any $T\in\mcT$ we may associate the permutation group $\Aut(M)\actson M$, where $M$ is the countable model of $T$. The previous theorem follows from a stronger property about approximations of random endomorphisms of $M$ by random automorphisms, which holds for infinite sets. We then show that this property is preserved under various permutation group-theoretic constructions, namely direct products and wreath products, as well as finite index supergroups.

The exact extent of the class of $T\in\mcT$ with approximately $\aleph_0$-categorical $(T^R)_P$ remains open, but it seems reasonable to propose the following:

\begin{conjecture}
If $T\in\mcD$, then $(T^R)_P$ is approximately $\aleph_0$-categorical.
\end{conjecture}

\subsection*{Acknowledgments} Tom\'as would like to thank Jordi Lopez-Abad for an early conversation related to approximate isomorphism of pairs of randomized infinite sets.

\section{Preliminaries}

We assume familiarity with continuous logic; the standard references are \cite{bbhu08,benusv10}. Classical first-order logic is seen as a particular case in the usual way. For simplicity we will consider only one-sorted languages.

\subsection{Elementary pairs and approximate isomorphisms}\label{ss:beautiful-pairs}

Let $\mcL$ be a countable metric language, and let $Q$ be an $\mcL$-theory. An \emph{elementary pair} of models of $Q$ is a couple $(M,N)$ where $N\preceq M$ and $M\models Q$. The class of elementary pairs of models of $Q$ is axiomatizable in the language $\mcL_P=\mcL\cup\{P\}$, where $P$ is a unary predicate that is interpreted as the distance function to the smaller model of the pair. Note that an $\mcL_P$-isomorphism between two elementary pairs $(M,N)$ and $(M',N')$ is the same as an isomorphism  between $M$ and $M'$ that maps $N$ onto $N'$.

\begin{defin}
Let $(M,N)$ and $(M',N')$ be two elementary pairs of models of $Q$. Given $\epsilon>0$, an \emph{$\epsilon$-isomorphism} between $(M,N)$ and $(M',N')$ is an $\mcL$-isomorphism $\sigma\colon M\to M'$ such that, for every $x\in N$,
$$|d(x,N)-d(\sigma(x),N')|\leq\epsilon.$$
The two pairs are \emph{approximately isomorphic} if there exist an $\epsilon$-isomorphism between them for every $\epsilon>0$.

An $\mcL_P$-theory $Q'$ whose models are elementary pairs of models of $Q$ is \emph{approximately $\aleph_0$-categorical} if every two separable models of $Q'$ are approximately isomorphic.
\end{defin}

Informally, two pairs are approximately isomorphic if they are $\mcL_P$-isomorphic \emph{up to arbitrarily small perturbations of the predicate $P$}. It is worth noting that the condition in the definition is equivalent to the inequality
$$d_H(\sigma(N),N')\leq\epsilon,$$
where $d_H$ denotes the Hausdorff distance between sets. Of course, for classical theories an $\epsilon$-isomorphism ($\epsilon<1$) is just an $\mcL_P$-isomorphism.

Approximate isomorphism has been studied in the context of continuous logic by Ben Yaacov in \cite{benYaacovOnpert}, by Ben Yaacov, Doucha, Nies, and Tsankov in \cite{msa}, and by the first author in \cite{hansonAppIso}. The formalisms in these papers are different in non-cosmetic ways, but our current notion of approximate isomorphism is sensible in all three.

We will be interested in so-called beautiful pairs  (\emph{belles paires}) of models, which where first studied by Poizat \cite{poiPaires}. We refer to Ben Yaacov's article \cite[\textsection 4]{benOnuniform} for a treatement in the metric setting.

\begin{defin}
An elementary pair $(M,N)$ is a \emph{beautiful pair} if $N$ is $\aleph_0$-saturated and $M$ is $\aleph_0$-saturated over $N$, i.e., every type over $N$ plus finitely many parameters from $M$ is realized in $M$. (One may wish to broaden the definition by replacing saturation with approximate saturation, as in \cite{benOnuniform}, but for our purposes this is unnecessary.)

The common $\mcL_P$-theory of all beautiful pairs of models of a theory $Q$ will be denoted by $Q_P$.
\end{defin}

If $Q$ is stable and satisfies that every sufficiently saturated model of $Q_P$ is a beautiful pair, then the class of beautiful pairs of models of $Q$ is considered to be well-behaved (called \emph{almost elementary} in \cite{bbhSFB14}, or \emph{weakly elementary} in \cite{benOnuniform}). This condition admits several equivalent (and quite different) formulations, which are the main subject of the seminal work of Poizat; see particularly \cite[Thm.~6]{poiPaires}, or \cite[Thm.~4.4]{benOnuniform} in the continuous setting. On the other hand, this condition is automatically satisfied if $Q$ is stable and $\aleph_0$-categorical, which will be the situation in the present paper.

Note that, in general, there is at most one \emph{separable} beautiful pair of models of $Q$, up to isomorphism. However, as mentioned in the introduction, even for $\aleph_0$-categorical $Q$ there might be more than one separable model of $Q_P$. The first example was given in \cite[Ex.~4.7]{bbhSFB14}, and another family of examples, at the core of the present paper, is discussed after Theorem~\ref{thm:beautiful-pair-of-TR} below. Notwithstanding, the following was highlighted in \cite{bbhSFB14}.

\begin{theorem}
If $T\in\mcT$, then $T_P$ is $\aleph_0$-categorical.
\end{theorem}

In the classical setting, $\aleph_0$-stability easily entails the existence of a beautiful pair of countable models.\footnote{This also holds in the metric setting (replacing ``countable'' by ``separable''), if one adopts the broader definition of beautiful pairs, i.e., with approximate saturation instead of saturation.} Hence in the previous theorem, the unique countable model of $T_P$ is a beautiful pair.

\subsection{Spaces of random variables} We lay down here some non-standard but convenient notations concerning spaces of random variables, and recall a useful measurable selector theorem.

Throughout this paper, $\Omega$ and $\Omega^2$ will denote the unit interval $[0,1]$ and the unit square $[0,1]^2$, respectively, endowed with their Borel $\sigma$-algebras $\mcB(\Omega)$ and $\mcB(\Omega^2)$, and their corresponding Lebesgue measures. In occasions we may write $\Omega^2=\Omega\times\Omega'$, where $\Omega'$ is a copy of $\Omega$. The measure of any of these spaces will be denoted by $\mu$ or, when it might help to avoid confusion, by $\mu_\Omega$, $\mu_{\Omega^2}$ or $\mu_{\Omega'}$.

If $F$ is some measurable condition involving random variables, $\llbracket F\rrbracket$ will denote the measurable set where the condition holds. The characteristic function of a measurable set $A$ will be denoted by $\chi_A$, or sometimes by $\chi A$ for readability.

Given a bounded, complete metric space $M$, we will denote by $M^\Omega$ the space $L(\Omega,M)$ of measurable functions from $\Omega$ to $M$, up to equality almost everywhere. Similarly, $M^{\Omega^2}\coloneqq L(\Omega^2,M)$. These spaces are equipped with the induced $L^1$ metric, $d(f,f')=\int d(f(\omega),f'(\omega))\rd\mu$, which is also complete. Typically, $M$ will be just a classical, discrete first-order structure, in which case $d(f,f')=\int\chi_{\llbracket {f\neq f'}\rrbracket}\rd\mu =\mu\llbracket {f\neq f'}\rrbracket$.

\begin{convention}
We will canonically identify $\mcB(\Omega)$ with a sub-$\sigma$-algebra of $\mcB(\Omega^2)$, namely the product of $\mcB(\Omega)$ with the trivial algebra on $\Omega'$. Similarly, we will identify $M^\Omega$ with a subset of $M^{\Omega^2}$, namely those random variables in $M^{\Omega^2}=L(\Omega\times\Omega',M)$ that depend only on the first coordinate.
\end{convention}

More generally, if $Z$ is a Polish space, then $Z^\Omega\coloneqq L(\Omega,Z)$ can be equipped with the Polish topology coming from any $L^1$ metric induced by a compatible bounded, complete metric on $Z$. Similarly for $Z^{\Omega^2}$. This will be used with $Z=\Aut(M)$ and $Z=\End(M)$, i.e., for considering the spaces of random automorphisms and random endomorphisms of a given structure.

We now recall the Jankov--von Neumann Uniformization Theorem, in the form that we will use it. See Kechris's book \cite[\textsection 18.A]{kechrisDST}.

\begin{theorem}\label{thm:JvN}
Let $Z$ be a Polish space, and let $E\subseteq \Omega\times Z$ be a Borel subset that projects onto $\Omega$. Then there is a Lebesgue measurable function $f\colon\Omega\to Z$ whose graph is contained in $E$. Idem with $\Omega$ replaced by $\Omega^2$.
\end{theorem}

As every Lebesgue measurable function is equal almost everywhere to a Borel measurable function, the theorem yields an element $\tilde{f}\in Z^\Omega$ with $(\omega,\tilde{f}(\omega))\in E$ almost everywhere.

\subsection{Pairs of randomizations}\label{ss:randomizations}

We will only be interested in randomizations of classical first-order structures, as studied in \cite{benkei}; we recall the basics below. For simplicity we opt here for the one-sorted presentation of randomizations (cf.~\cite[Rmk.~2.12]{benkei}), in which the \emph{event sort} is not explicitly included in the language (it becomes an imaginary sort).

Let $\mcL$ be a countable classical language. The corresponding \emph{randomization language} $\mcL^R$ is a metric language containing, for each $\mcL$-formula $\varphi$ in $n$ free variables, a predicate symbol $\mbP[\varphi]$ of arity $n$ that is 1-Lipschitz in each variable. Given a countable $\mcL$-structure $M$, we see $M^\Omega$ (and $M^{\Omega^2}$) as an $\mcL^R$-structure by interpreting, for every $n$-tuple $f$ from $M^\Omega$,
$$\mbP[\varphi](f) \coloneqq \mu\llbracket\varphi(f)\rrbracket = \mu\{\omega\in\Omega : M\models\varphi(f(\omega))\}.$$

For the rest of the section let us fix a theory $T\in\mcT$. The \emph{randomized theory} $T^R$ is the $\mcL^R$-theory of the structure $M^\Omega$, where $M$ is a countable model of $T$. This clearly does not depend on the particular countable model $M$ or the particular atomless Lebesgue space $\Omega$, since changing them results in an isomorphic $\mcL^R$-structure.

\begin{theorem}[\cite{benkei}]
The metric theory $T^R$ is $\aleph_0$-categorical and $\aleph_0$-stable. Moreover, $T^R$ has quantifier elimination.
\end{theorem}

Given countable models $N\preceq M$ of $T$, the structures $N^\Omega$, $M^\Omega$ and $M^{\Omega^2}$ are separable models of $T^R$ and so, by quantifier elimination, the inclusions $N^\Omega\subseteq M^\Omega\subseteq M^{\Omega^2}$ are elementary.

The following was observed in \cite[Rmk.~5.12]{ibaRando}\footnote{Note the shift in the numbering of sections and results in the published version of \cite{ibaRando} and the available preprint versions.}.

\begin{theorem}\label{thm:beautiful-pair-of-TR}
Let $(M,N)$ be a beautiful pair of countable models of $T$. Then $(M^{\Omega^2},N^\Omega)$ is a beautiful pair of models of $T^R$.
\end{theorem}

The aforesaid remark in \cite{ibaRando} sketches how to derive this result by modifying the proof of a more complicated result. For the convenience of the reader, we give a direct and alternative proof of this theorem in the Appendix.

As mentioned in the introduction, there are other separable models of $(T^R)_P$, starting with $(M^{\Omega^2},M^\Omega)$, which is non-beautiful if $M$ is infinite. A useful description of all separable models of $(T^R)_P$ up to isomorphism was also given in~\cite{ibaRando}, which we recall next.

Given a separable structure $M$, let $\Aut(M)$ denote its automorphism group. This is a Polish group under the topology of pointwise convergence. Similarly, we let $\End(M)$ be the semigroup of elementary self-embeddings of $M$, which we call \emph{endomorphisms} for short. It is also Polish with the topology of pointwise convergence.

The semigroup $\End(M)^\Omega$ of random endomorphisms of $M$ acts by endomorphisms on $M^\Omega$, by the formula $(\hat{h}f)(\omega)=\hat{h}(\omega)(f(\omega))$ for $\hat{h}\in\End(M)^\Omega$ and $f\in M^\Omega$. (See \cite[Thm.~3.8, Cor.~3.11]{ibaRando} for a description of $\Aut(M^\Omega)$ and $\End(M^\Omega)$.) Similarly, $\End(M)^{\Omega^2}$ can be seen as a subsemigroup of $\End(M^{\Omega^2})$, and $\Aut(M)^{\Omega^2}$ as a subgroup of $\Aut(M^{\Omega^2})$. The following result is \cite[Thm.~5.6]{ibaRando}.

\begin{theorem}\label{thm:models-of-TRP}
Let $M\models T$ be the countable model of $T$, and let $\hat{h}\in \End(M)^{\Omega^2}\subseteq\End(M^{\Omega^2})$. Then $(M^{\Omega^2},\hat{h}(M^\Omega))$ is a model of $(T^R)_P$. Moreover, every separable model of $(T^R)_P$ is of this form, up to isomorphism.
\end{theorem}

\begin{rem}\ 
\begin{enumerate}
\item For every pair as in the theorem, the map $\hat{h}$ gives an elementary embedding of $(M^{\Omega^2},M^\Omega)$ into that pair. Thus $(M^{\Omega^2},M^\Omega)$ is the prime model of $(T^R)_P$. The fact that $\hat{h}$ is elementary in the language of pairs can be deduced, for instance, from the fact that $(T^R)_P$ eliminates quantifiers in the language expanded with a canonical base map: see \cite[Cor.~4.6]{benOnuniform}.

\item Adapting results and proofs from the literature (cf.~\cite[Thm.~9]{poiPaires}, \cite[Thm~4.4]{benOnuniform}), one should be able to show that the theory $(T^R)_P$ is $\aleph_0$-stable and that the pair $(M^{\Omega^2},N^\Omega)$ from Theorem~\ref{thm:beautiful-pair-of-TR} is precisely its separable, approximately (possibly exactly) $\aleph_0$-saturated model. We shall not need these facts.
\end{enumerate}
\end{rem}

\section{Randomized non-disintegrated minimal sets}\label{s:counterexamples}

Let $M$ be the countable model of some $T\in\mcT$. We recall that a \emph{strictly minimal set} of $M$ is a strongly minimal ($M$-definable) set $Q$ in $M^\eq$ such that $\acl(a)\cap Q=\{a\}$ for every $a\in S$. The algebraic closure operator defines a geometry on the strictly minimal set $Q$, which can be either \emph{disintegrated} (equivalently, $Q$ is an indiscernible set) or the geometry of an affine or projective space over a finite field. See \cite[Thm.~2.1]{cheharlac}, as well as their Appendix~I.

As in the introduction, let $\mcD$ be the class of those theories $T\in\mcT$ such that all corresponding strictly minimal sets are disintegrated.

\begin{theorem}
Suppose $T\in\mcT\setminus\mcD$. Then the separable beautiful pair of models of $T^R$ is not approximately isomorphic to the prime model of $(T^R)_P$. Thus, $(T^R)_P$ is not approximately $\aleph_0$-categorical.
\end{theorem}
\begin{proof}  
Let $Q$ be a strictly minimal set of the countable model $M$ of $T$. Say $Q$ is the quotient of the set defined by a formula $\varphi(\bar{x},\bar{a})$ by the equivalence relation defined by a formula $E(\bar{x},\bar{y},\bar{a})$. Let $n=|\bar{x}|=|\bar{y}|$.

As mentioned above, the geometry of $Q$ is either that of an infinite dimensional affine space over a finite field or that of an infinite dimensional projective space over a finite field. These geometries all have the property that for any $\delta > 0$ there is a $k\in\mbN$ such that:
\begin{equation}\label{eq:affine-proj-delta-k}\tag{*}
\text{if $A$ and $B$ are finite closed sets with $A \supseteq B$ and $\dim(A/B)\geq k$, then $|B|\leq \delta |A|$.}
\end{equation}

Fix a sequence $\{Q_i\}_{i\in\mbN}$ of non-empty finite closed subsets of $Q$ such that $Q_i \subseteq Q_{i+1}$ and $\bigcup_{i\in\mbN}Q_i = Q$. Clearly it must be the case that $\dim(Q_i)$ grows unboundedly.

Fix $\epsilon>0$, and suppose the pair $(M^{\Omega^2},M^\Omega)$ is $\epsilon$-isomorphic to the separable beautiful pair of models of $T^R$. Hence there is an elementary submodel $R\preceq M^{\Omega^2}$ such that $M^{\Omega^2}$ is $\aleph_0$-saturated over $R$ and $d_H(R,M^\Omega) \leq \epsilon$. Let $\epsilon'>0$ be arbitrary. Fix some $\bar{c}/E\in Q_0$, and choose $\bar{c}' \in R^n$ with $d(\bar{c},\bar{c}') \leq \epsilon+\epsilon'$ (we use the max-distance on finite tuples), where we are conflating $\bar{c}\in M^n$ with the constant function on $\Omega^2$ taking on the value $\bar{c}$. Let $C = \{\alpha \in \Omega^2 : \bar{c} = \bar{c}'(\alpha)\}$. Note that by construction, we have $\mu(C) \geq  1 - n(\epsilon+\epsilon')$.

We take $R_0 \subseteq R$ a countable dense subset such that $\bar{c}'\in (R_0)^n$. Let each element of $R_0$ be represented by a concrete measurable function from $\Omega^2$ to $M$. For any $\alpha \in C$, let
\[
R(\alpha) = \acl_Q\{f(\alpha): f \in R_0\},
\]
where, given $S\subseteq M$, $\acl_Q(S)$ is the set of elements of the quotient $Q$ in the algebraic closure of the $E$-classes of tuples in $S^n\cap \varphi(M^n,\bar{a})$. We claim that for any $k\in\mbN$, $\dim(Q/R(\alpha))\geq k$ almost everywhere. Indeed, since $M^{\Omega^2}$ is saturated over $R\bar{a}$, we can find $\{f_j\}_{j<k}\subseteq \varphi(M^n,\bar{a})^{\Omega^2}$ such that, for every finite subset $F\subseteq R$, the first-order condition
\[
\mu\big\llbracket\dim(\acl_Q(F\cup\{f_j\}_{j<k})/\acl_Q(F)) \geq k \big\rrbracket = 1
\]
holds. The claim follows.

Now let $\delta > 0$ be arbitrary, and find $k$ satisfying the condition (\ref{eq:affine-proj-delta-k}) stated at the beginning of the proof. For each $i\in\mbN$, we consider:
\[
C_i = \{\alpha \in C : \dim(Q_i / (Q_i \cap R(\alpha))) \geq k\}.
\]
Note that each $C_i$ is a measurable set. As $C_i\subseteq C$, for every $\alpha\in C_i$ we have $\bar{c}/E\in Q_i\cap R(\alpha)$, so this intersection is non-empty. Thus, since the geometry is locally modular, we have $C_{i+1} \supseteq C_i$ for every $i\in\mbN$. Indeed, if we take $\alpha\in C_i$ and denote $\tilde{R}_i(\alpha)=Q_i\cap R(\alpha)$, then:
\begin{align*}
k & \leq \dim(Q_i / \tilde{R}_i(\alpha))) = \dim(Q_i)-\dim(Q_i\cap \tilde{R}_{i+1}(\alpha)) = \dim(Q_i\cup \tilde{R}_{i+1}(\alpha))- \dim(\tilde{R}_{i+1}(\alpha)) \\ & \leq \dim(Q_{i+1})-\dim(\tilde{R}_{i+1}(\alpha)) = \dim(Q_{i+1}/\tilde{R}_{i+1}(\alpha)).
\end{align*}
Since $R(\alpha)$ has codimension at least $k$ almost everywhere, we have $C=\bigcup_{i\in\mbN}C_i$ up to measure zero. In particular, there is $i\in\mbN$ such that $\mu(C_i) > (1-\delta)\mu(C)$. Consider the quantity
  \[
   I = \frac{1}{|Q_i|} \int_C |\tilde{R}_i(\alpha)|d\mu(\alpha) \leq \int_{C_i} \frac{|\tilde{R}_i(\alpha)|}{|Q_i|}d\mu(\alpha) + \mu(C\setminus C_i).
  \]
By choice of $k$, the integral over $C_i$ can be no larger than $\delta\mu(C_i)$, and by choice of $i$, $\mu(C\setminus C_i)<\delta\mu(C)$. Hence $I < 2\delta$. On the other hand, we note that
\[
I = \frac{1}{|Q_i|} \sum_{q \in Q_i} \mu\{\alpha \in C : q \in R(\alpha)\}.
\]
So there must be some $\bar{b}/E \in Q_i$ such that $\mu\{\alpha \in C : \bar{b}/E \in R(\alpha)\} < 2\delta$. Let us identify $\bar{b}\in M^n$ with the tuple in $(M^{\Omega})^n = (M^n)^\Omega$ constantly equal to $\bar{b}$.

We have that, for any $\bar{f}\in (R_0)^n$,
  \[
  d(\bar{f},\bar{b}) \geq \frac{1}{n}\mu\llbracket \bar{f}\neq \bar{b}\rrbracket \geq \frac{1}{n}\mu\{\alpha\in C: \bar{b}/E\notin R(\alpha)\} > \frac{1}{n}(\mu(C) - 2\delta) \geq \frac{1}{n}(1 - 2\delta)-\epsilon -\epsilon',
  \]
  but this results in a contradiction with the condition $d_H(R,M^\Omega)\leq\epsilon$ whenever $\epsilon < \frac{1}{n}(1 - 2\delta)-\epsilon -\epsilon'$. Since $\delta$ and $\epsilon'$ were arbitrary, we conclude that any $\epsilon$-isomorphism between $(M^{\Omega^2},M^\Omega)$ and the separable beautiful pair of models of $T^R$ has to satisfy:
\[
\epsilon \geq \frac{1}{2n}.
\]
The result follows. We may remark that if the geometry of $Q$ is modular then it is unnecessary to consider the set $C$, and the above bound becomes $\epsilon\geq\frac{1}{n}$.
\end{proof}

\section{Approximating endomorphisms by automorphisms}\label{s:positive-results}

As in the previous section, we let $T$ be some theory in $\mcT$ and $M$ be its countable model.

\begin{defin}
We will say that $h\in \End(M)$ is \emph{approximable by automorphisms} if for every $\epsilon>0$ there are $n\in\mbN$ and $g_0,\dots,g_{n-1}\in\Aut(M)$ such that, for every $a\in M$,
$$|\{i<n : g_i(a)\neq h(a)\}|<\epsilon n.$$
We will say $M$ has \emph{approximable endomorphisms} if every $h\in\End(M)$ is approximable by automorphisms.
\end{defin}

\begin{prop}\label{p:one-approximable-endo}
If $h\in\End(M)$ is approximable by automorphisms, then $(M^{\Omega^2},M^\Omega)$ and $(M^{\Omega^2},h(M)^\Omega)$ are approximately isomorphic.
\end{prop}
\begin{proof}
Given $\epsilon>0$, we choose $g_0,\dots,g_{n-1}$ as in the definition. We then define $\hat{g}\in \Aut(M)^{\Omega^2}\subseteq \Aut(M^{\Omega^2})$ by:
$$\hat{g}(\omega,\omega')=g_i\text{ iff }\omega'\in \left[\frac{i}{n},\frac{i+1}{n}\right).$$
We also consider the elementary self-embedding $\hat{h}$ of $M^{\Omega^2}$ defined by $\hat{h}(f)=h\circ f$, and observe that $h(M)^\Omega=\hat{h}(M^\Omega)$.

Now, given any $f\in M^\Omega\subseteq M^{\Omega^2}$, we have:
$$d(\hat{g}(f),\hat{h}(f)) = \mu_{\Omega^2}\llbracket \hat{g}(f)\neq \hat{h}(f)\rrbracket =  \sum_{i<n} \frac{1}{n} \mu_\Omega\llbracket g_if\neq hf\rrbracket = \int_\Omega\frac{1}{n}\sum_{i<n}\chi{\llbracket g_if\neq hf\rrbracket}\rd\omega < \epsilon,$$
since $\sum_{i<n}\chi{\llbracket g_if\neq hf\rrbracket}(\omega) = |\{i<n : g_i(f(\omega))\neq h(f(\omega))\}|<\epsilon n$. Thus, in particular,
$$d_H(\hat{g}(M^\Omega),h(M)^\Omega) < \epsilon,$$
and we see that $\hat{g}$ is an $\epsilon$-isomorphism between the pairs of the statement.
\end{proof}

Our way of establishing approximate $\aleph_0$-categoricity of $(T^R)_P$ in certain cases will be through the following seemingly stronger property.

\begin{defin}\label{def:app-random-endo}
We will say that $M$ has \emph{approximable random endomorphisms} if for every $\hat{h}\in\End(M)^{\Omega^2}$ and $\epsilon>0$ there is $\hat{g}\in\Aut(M)^{\Omega^2}$ such that, for every $f\in M^\Omega$, $d(\hat{g}(f),\hat{h}(f))<\epsilon$.
\end{defin}

Note well that the approximation condition only concerns the elements of the submodel $M^\Omega\subseteq M^{\Omega^2}$.

\begin{lem}\label{l:app-ran-end-to-TRP}
If $M$ has approximable random endomorphisms, then $(T^R)_P$ is approximately $\aleph_0$-categorical.
\end{lem}
\begin{proof}
If $\hat{h}$ and $\hat{g}$ are as in Definition~\ref{def:app-random-endo}, then $d_H(\hat{g}(M^\Omega),\hat{h}(M^\Omega))\leq\epsilon$, so $\hat{g}$ is an $\epsilon$-isomorphism between $(M^{\Omega^2},M^\Omega)$ and $(M^{\Omega^2},\hat{h}(M^\Omega))$. Now we apply Theorem~\ref{thm:models-of-TRP}.
\end{proof}

For the next result we will consider the natural left action $\Aut(M)\actson \End(M)$ and the corresponding orbit set, $O(M)\coloneqq\Aut(M)\backslash\End(M)$.

\begin{prop}\label{prop:TRP-approx-a0-cat-sufficient-conditions}
Suppose $M$ has approximable endomorphisms, and that the orbit set $O(M)$ is countable. Then $M$ has approximable random endomorphisms.
\end{prop}
\begin{proof}
Let $\{h_k\}_{k\in\mbN}$ be a family of endomorphisms of $M$ such that $\End(M)=\bigcup_{k\in\mbN}\Aut(M)h_k$. Given $\hat{h}\in\End(M)^{\Omega^2}$, we consider the sets
$$E_k=\{(x,g)\in \Omega^2 \times \Aut(M) : \hat{h}(x) = gh_k \},\ E=\bigcup_{k\in\mbN}E_k.$$
By choice of the $h_k$, the projection of $E$ to the first coordinate is all of $\Omega^2$. On the other hand, $E$ is a Borel set, as each set $E_k$ is the preimage of the closed set
$$\{(h,g)\in\End(M)\times\Aut(M):h(a)=g(h_k(a))\ \text{for all}\ a\in M\}$$
by the product map $\hat{h}\times\mathrm{id}\colon \Omega^2\times \Aut(M)\to \End(M)\times\Aut(M)$. We may therefore apply the Jankov--von Neuman theorem (Thm.~\ref{thm:JvN}) to obtain an element $\hat{g}\in \Aut(M)^{\Omega^2}$ such that for almost every $x$, the endomorphism $\hat{g}\inv\hat{h}(x)$ is one of the endomorphisms $h_k$.

We denote $\hat{h}' =\hat{g}\inv\hat{h}\in\End(M)^{\Omega^2}$, and we aim now to find $\hat{g}'\in\Aut(M)^{\Omega^2}$ that satisfies the condition of Definition~\ref{def:app-random-endo} with $\hat{h}'$. For each $k\in\mbN$, let $\Omega^2_k$ be the set of $x\in\Omega^2$ such that $\hat{h}'(x)=h_k$. Thus $\{\Omega^2_k\}_{k\in\mbN}$ is a partition of $\Omega^2=\Omega\times\Omega'$.

It is then enough to repeat the argument of Proposition~\ref{p:one-approximable-endo} inside each set $\Omega_k^2$. Namely, given $\epsilon>0$ we choose $g_0^k,\dots,g_{n_k-1}^k\in \Aut(M)$ such that
$$|\{i<n_k : g^k_i(a)\neq h_k(a)\}|<\epsilon n_k$$
for every $a\in M$. We then choose a finite measurable partition $\{A_i^k\}_{i<n_k}$ of $\Omega^2_k$ such that for almost every $\omega\in\Omega$ and every $i<n_k$, the measures of the slices $(A_i^k)_\omega\subseteq (\Omega^2_k)_\omega\subseteq \Omega'$ satisfy:
$$\mu_{\Omega'}((A_i^k)_\omega) = \frac{1}{n_k}\mu_{\Omega'}((\Omega^2_k)_\omega).$$ Finally, we define $\hat{g}'\in \Aut(M)^{\Omega^2}\subseteq \Aut(M^{\Omega^2})$ by:
$$\hat{g}'(x)=g^k_i\text{ iff }x\in A^k_i.$$
Hence for every $f\in M^\Omega$,
\begin{align*}
d(\hat{g}'(f),\hat{h}'(f)) & = \sum_{k\in\mbN}\int_{\Omega^2_k} \chi{\llbracket \hat{g}'(f)\neq \hat{h}'(f)\rrbracket}(x)\rd x \\
& = \sum_{k\in\mbN}\int_\Omega \sum_{i<n_k} \int_{(A^k_i)_\omega} \chi{\llbracket \hat{g}'(f)\neq \hat{h}'(f)\rrbracket}(\omega,\omega')\rd\omega' \rd\omega \\
& = \sum_{k\in\mbN}\int_{\Omega} \frac{1}{n_k}\mu_{\Omega'}((\Omega^2_k)_\omega) \sum_{i<n_k}\chi{\llbracket g^k_if\neq h_kf\rrbracket}(\omega)\rd\omega \\
& < \sum_{k\in\mbN}\frac{1}{n_k}\mu(\Omega^2_k)\cdot \epsilon n_k = \epsilon.
\end{align*}
Coming back to $\hat{h}$, we have $d(\hat{g}\hat{g}'(f),\hat{h}(f))<\epsilon$ for every $f\in M^\Omega$. This concludes the proof.
\end{proof}

The hypothesis that the action $\Aut(M)\actson\End(M)$ has countably many orbits is extremely restrictive. It fails already for the structure consisting of an equivalence relation with infinitely many infinite classes. In fact, the proposition is justified mainly by one example, that of the pure infinite set. Indeed, if $M$ is a countable set with no structure, then $O(M)$ is clearly countable, and the following proposition shows that $M$ has approximable endomorphisms.

\begin{prop}
Let $X$ be a set. For every injective map $\tau\colon X\to X$ and every $n\in\mbN$ there exist bijections $\sigma_i\colon X\to X$, $i<n$, such that for each $x\in X$, $|\{i<n:\sigma_i(x)\neq \tau(x)\}|\leq 1$.
\end{prop}
\begin{proof}
We partition $X$ into the orbits and semi-orbits of $\tau$ (an orbit being a set $\{x_j\}_{j\in\mbZ}$ where $\tau x_j = x_{j+1}$ for every $j\in\mbZ$; a semi-orbit being a set $\{x_j\}_{j\in\mbN}$ with $\tau x_j = x_{j+1}$ for every $j\in\mbN$ and such that $x_0$ has no preimage by $\tau$). On the orbits we set each $\sigma_i$, $i<n$, to be equal to $\tau$. Given a semi-orbit $\{x_j\}_{j\in\mbN}$, we choose $n$  increasing sequences $j^i=\{j^i_k\}_{k\in\mbN}\subseteq\mbN$, $i<n$, with the property that $j^i_k\neq j^{i'}_{k'}$ whenever $(i,k)\neq (i',k')$. Then we define $\sigma_i$ on the semi-orbit by the rule:
\[
\sigma_i(x_j)=\begin{cases}
\tau(x_j) = x_{j+1} & \text{ if } j\neq j^i_k\ \forall k\in\mbN \\
x_{j^i_{k-1}} & \text{ if } j=j^i_k,
\end{cases}
\]
with the convention that $j^i_{-1}=0$.

The resulting maps $\sigma_i\colon X\to X$ are bijections with the desired property.
\end{proof}

As a corollary  of the previous propositions and lemma we obtain:

\begin{theorem}
Let $T$ be the theory of pure infinite sets. Then $(T^R)_P$ is approximately $\aleph_0$-categorical.
\end{theorem}

Next we will show that the property of having approximable random endomorphisms is preserved under a number of constructions. For this we prefer to use the language of permutation groups. When we say that $G\actson X$ is a permutation group, we understand that $X$ is a countable set and $G$ is a closed subgroup of $\Sym(X)$. As is well-known, an $\aleph_0$-categorical structure $M$ can be identified with the \emph{oligomorphic} permutation group $\Aut(M)\actson M$. In that case, the semigroup $\End(M)$ is just the left-completion of the group $\Aut(M)$. It thus makes sense to say that an oligomorphic permutation group has or has not approximable random endomorphisms.

\begin{defin}
Let $G\actson M$ and $H\actson N$ be permutation groups, which we may denote simply by $G$ and $H$.
\begin{enumerate}
\item The direct product of $G$ and $H$ is the permutation group $G\times H\actson M\sqcup N$, where $(g,h)a= ga$ if $a\in M$ and $(g,h)b=hb$ if $b\in N$. We denote it simply by $G\times H$.
\item The wreath product of $G$ and $H$ is the permutation group $H\ltimes G^N \actson N \times M$, where $(h,g)(b,a) = (hb,g(b)a)$. We denote it by $G\wr H$. 
\item $G$ is a finite index supergroup of $H$ if $M=N$, $H\leq G$ and $[G:H]<\infty$.
\end{enumerate}
\end{defin}

\begin{rem}
In the semidirect product $H\ltimes G^N$, it is convenient to see $H$ as acting on $G^N$ on the \emph{right}, i.e., by $(g\cdot h)(b)=g(hb)\in G$, so that the group law is $(h,g)(h',g')=(hh',(g\cdot h')g')$. This is essential for the definition of the semidirect product if one considers \emph{semigroups} instead of groups, as is the case when passing from $G\wr H$ to its left-completion $\ov{G\wr H}$. The latter can be identified with $\ov{H}\ltimes\ov{G}^N$.
\end{rem}

\begin{prop}
Let $G\actson M$ and $H\actson N$ be oligomorphic permutation groups.
\begin{enumerate}
\item If $G$ and $H$ have approximable random endomorphisms, then so do $G\times H$ and $G\wr H$.
\item If $G$ is a finite index supergroup of $H$ and $H$ has approximable random endomorphisms, then so does $G$.
\item If $G$ has approximable random endomorphisms and $G\simeq H$ as topological groups, then $H$ has approximable random endomorphisms. 
\end{enumerate}
\end{prop}
\begin{proof}\ 
\begin{enumerate}[wide]
\item Suppose that $G$ and $H$ have approximable random endomorphisms. For any Polish group $A$, let $\ov{A}$ be the left-completion of $A$. We have that $\ov{G\times H} = \ov{G}\times \ov{H}$ and so $\ov{G\times H}^{\Omega^2} = \ov{G}^{\Omega^2} \times \ov{H}^{\Omega^2}$. Given $(\hat{g},\hat{h}) \in \ov{G}^{\Omega^2} \times \ov{H}^{\Omega^2}$ and $\epsilon > 0$, we can by assumption find $\hat{g}' \in G^{\Omega^2}$ and $\hat{h}' \in H^{\Omega^2}$ such that for any $f_0 \in M^\Omega$ and $f_1 \in N^\Omega$, $d(\hat{g}(f_0),\hat{g}'(f_0)) < \epsilon/2$ and $d(\hat{h}(f_1),\hat{h}'(f_1)) < \epsilon/2$. An element $f$ of $(M \sqcup N)^{\Omega}$ can be specified (up to measure zero) by a set $B\in\mcB(\Omega)$ together with measurable functions $f_0 : \Omega \to M$ and $f_1 : \Omega \to N$ by letting $f$ be $f_0$ on $B$ and $f_1$ on $\Omega\setminus B$. We then have that
    \begin{align*}
      d((\hat{g},\hat{h})(f),(\hat{g}',\hat{h}')(f)) &\leq d(\hat{g}(f_0),\hat{g}'(f_0)) + d(\hat{h}(f_1),\hat{h}'(f_1)) < \epsilon.
    \end{align*}
    So since we can do this for any endomorphism $(\hat{g},\hat{h})$ and $\epsilon > 0$, we have that $G \times H$ has approximable random endomorphisms.

    For the wreath product, we have that $\ov{G\wr H}^{\Omega^2} =  \ov{H}^{\Omega^2} \ltimes (\ov{G}^N)^{\Omega^2}$. An element $(\hat{h},\hat{g})$ (with $\hat{h} \in \ov{H}^{\Omega^2}$ and $\hat{g} \in (\ov{G}^N)^{\Omega^2}$) acts on $(f_0,f_1) \in N^{\Omega^2} \times M^{\Omega^2} = (N \times M)^{\Omega^2}$ by
      \[
        (\hat{h},\hat{g})(f_0,f_1)(\bar\omega) = (\hat{h}(\bar\omega)f_0(\bar\omega),\hat{g}(\bar\omega)[f_0(\bar\omega)]f_1(\bar\omega))\in N\times M
      \]
      for $\bar\omega \in \Omega^2$.  Fix $(\hat{g},\hat{h}) \in \ov{G\wr H}^{\Omega^2}$. We will now think of $(\ov{G}^N)^{\Omega^2}$ as $(\ov{G}^{\Omega^2})^N$, so that for any $b \in N$, $\hat{g}[b]$ is an element of $\ov{G}^{\Omega^2}$.
      
	Let $(b_i)_{i \in \mbN}$ be a fixed enumeration of $N$. By assumption, $H$ and $G$ have approximable random endomorphisms. Therefore, for any $\epsilon > 0$, we can find $\hat{h}' \in H^{\Omega^2}$ and $\hat{g}' \in (G^{\Omega^2})^N$ such that, for any $f_0 \in N^\Omega$, $d(\hat{h}(f_0),\hat{h}'(f_0)) < \frac{1}{2}\epsilon$ and, for any $f_1 \in M^\Omega$ and $i \in \mbN$, $d(\hat{g}[b_i](f_1),\hat{g}'[b_i](f_1)) < 2^{-i-2}\epsilon$.

      Fix some element $f = (f_0,f_1) \in (M\times N)^{\Omega} = M^{\Omega}\times N^{\Omega}$. A sufficient condition for $(\hat{h},\hat{g})(f_0,f_1)(\bar\omega) = (\hat{h}',\hat{g}')(f_0,f_1)(\bar\omega)$ is to have $\hat{h}(\bar\omega)f_0(\bar\omega) = \hat{h}'(\bar\omega)f_0(\bar\omega)$ and $\hat{g}[b_i](\bar\omega)f_1(\bar\omega)=\hat{g}'[b_i](\bar\omega)f_1(\bar\omega)$ for all $i\in\mbN$. By construction, the set
      \[
        A = \{\bar\omega \in \Omega^2 : \hat{h}(\bar\omega)(f_0(\bar\omega)) = \hat{h}'(\bar\omega)(f_0(\bar\omega))\}
      \]
      has measure greater than $1-\frac{1}{2}\epsilon$. Furthermore, by construction, the set
      \[
B = \{\bar\omega \in \Omega^2 : (\forall i \in \mbN)\hat{g}'[b_i](\bar{\omega})f_1(\bar{\omega}) = \hat{g}[b_i](\bar{\omega})f_1(\bar{\omega})\}
\]
has measure greater than $1-\frac{1}{2}\epsilon$. Therefore $\mu(A\cap B) > 1 -\epsilon$, and thus $d((\hat{g},\hat{h})(f),(\hat{g}',\hat{h}')(f)) < \epsilon$.

Since we can do this for any $(\hat{g},\hat{h}) \in \ov{G \wr H}^{\Omega^2}$ and any $\epsilon > 0$, we have that the structure on $N\times M$ corresponding to $G \wr H$ has approximable random endomorphisms.

\item Assume that $G$ is a finite index supergroup of $H$, a group with approximable random endomorphisms. Let $g_1,\dots,g_n$ be elements of $G$ such that $g_1 H,\dots,g_n H$ are the cosets of $H$ in $G$. It is easy to see that $\ov{G} = g_1 \ov{H}\sqcup \dots \sqcup g_n \ov{H}$. Let $\hat{g}$ be an element of $\ov{G}^{\Omega^2}$. We can find a measurable function $\hat{i} : \Omega^2 \to \{g_1,\dots,g_n\}$ and an element $\hat{h} \in \ov{H}^{\Omega^2}$ such that $\hat{g} = \hat{i}\hat{h}$. Fix $\epsilon > 0$, and let $\hat{h}' \in H^{\Omega^2}$ be such that for every $f \in M^{\Omega}$, $d(\hat{h}(f),\hat{h}'(f)) < \epsilon$. We then have that for every $f \in M^\Omega$, $d(\hat{g}(f),\hat{i}\hat{h}'(f)) < \epsilon$ as well. Since we can do this for any $f \in M^\Omega$ and $\epsilon > 0$, we have that $G$ has approximable random endomorphisms.
    
\item Suppose $G$ and $H$ are isomorphic as topological groups. As is well-known (cf.\ \cite{ahlzie84}), since $G$ and $H$ are oligomorphic, this means that the structures associated to them are bi-interpretable, so that in particular there is a $G$-invariant equivalence relation $E$ on some $G$-invariant subset $D\subseteq M^n$ such that $G\actson D/E$ and $H\actson N$ are isomorphic as left actions (i.e., as permutation groups, up to identifying $G$ with a subgroup of $\Sym(D/E)$). Thus, to see that $H$ has approximable random endomorphisms, it suffices to check that for every $\epsilon>0$ and $\hat h\in \ov{G}^{\Omega^2}$ there is $\hat g\in G^{\Omega^2}$ such that for every $f\in D^\Omega$ we have:
    \[\mu\llbracket \hat{g}(f)\text{ is not $E$-related to }\hat{h}(f)\rrbracket < \epsilon.\]
This holds clearly if $G$ has approximable random endomorphisms.\qedhere
\end{enumerate}
\end{proof}

One can check that if $M$ and $N$ are countable models of theories in $\mcT$, then the theories corresponding to the permutation groups $\Aut(M)\times \Aut(N)$ and $\Aut(M)\wr \Aut(N)$ are also in~$\mcT$: the case of the direct product is easy, and for the wreath product notice that the structure associated to $\Aut(M)\wr \Aut(N)$ is interpretable in the structure associated to $\Aut(M)\times \Aut(N)$. Similarly, if the theory of $N$ is in $\mcT$ and $\Aut(M)$ is a finite (or infinite) index supergroup of $\Aut(N)$, then the theory of $M$ is in $\mcT$, because $M$ is a reduct of $N$.

Starting from the pure countable infinite set and using the previous constructions (as well as bi-intepretations), we thus get, using Lemma~\ref{l:app-ran-end-to-TRP}, a rich family of examples of theories $T\in\mcT$ for which $(T^R)_P$ is approximately $\aleph_0$-categorical.

For instance, according to results of Lachlan~\cite{lachlanTree}, and more explicitly Bodor \cite[Thm.~3.44]{bodorClass}, the class of structures that can be obtained from the trivial structure by taking direct products, wreath products of the form $G\wr S_\infty$, and finite index supergroups, is exactly the class of $\aleph_0$-categorical \emph{monadically stable} structures. So these all have approximable random endomorphisms. Using arbitrary wreath products and bi-interpretations, we see that many further basic examples also have approximable random endomorphisms (e.g., sets with several equivalence relations with infinitely many classes whose classes intersect non-trivially: these are not monadically stable).

As mentioned in the introduction, it is natural in view of our results to ask whether the class $\mcT_0$ of theories $T\in\mcT$ such that $(T^R)_P$ is approximately $\aleph_0$-categorical is precisely the class $\mcD$ of theories with disintegrated strictly minimal sets. We end with two subsidiary questions.

\begin{question}\ 
\begin{enumerate}
\item Is the class $\mcT_0$ closed under interpretations?
\item Is the class $\mcT_0$ closed under finite covers?
\end{enumerate}
\end{question}

Following our previous strategy, it might also be useful to consider these questions for the property of having approximable random endomorphisms.

\section*{Appendix}

We give a proof of Theorem~\ref{thm:beautiful-pair-of-TR} based on the idea of the argument of \cite[Thm.~4.1]{benkei}. Here it will be practical to consider formulas involving parameters from the event sort, i.e., expressions such as $\mbP([\varphi(x)]\sqcap B)$ where $B\in\mcB(\Omega^2)$. Any such expression can be replaced by a formula of the form $\mbP[\varphi(x)\wedge (f=f')]$, for appropriate parameters $f,f'$ from the main sort.

Let $(M,N)$ be a beautiful pair of countable models of $T\in\mcT$. Let $t\in (M^{\Omega^2})^n$ be a finite tuple, and take a type $p\in S_1(N^\Omega t)$. We want to show that $p$ is realized in $M^{\Omega^2}$. Let $C\subseteq M^n$ be the image of $t$ (seen as a Borel map $t\colon \Omega^2\to M^n$), which is a countable set. For each $c\in C$ we denote $X_c=S_1(Nc)$. We also let $A_c=t\inv(c)\in \mcB(\Omega^2)$. In addition, we fix an extension $p'\in S_1(M^{\Omega^2})$ of $p$ to a type over $M^{\Omega^2}$.

For each $c\in C$ and $q\in X_c$, let $\{\varphi_q^i(x)\}_{i\in\mbN}$ be a sequence of consistent formulas contained in $q$ with the property that $\varphi_q^{i+1}(x) \vdash \varphi_q^i(x)$ for each $i\in\mbN$ and $\{\varphi_q^i(x)\}_{i\in\mbN}\vdash q(x)$. For any $B\in\mcB(\Omega)$, let
\[
\nu_q(B) = \lim_{i \to \infty} \mbP([\varphi_q^i(x)] \sqcap A_c \sqcap B)^{p'}.
\]
Note that since $\nu_q$ is the limit of a monotonically decreasing sequence of measures on $\mcB(\Omega)$ that are absolutely continuous with regards to $\mu$, we have that $\nu_q$ is itself a measure on $\mcB(\Omega)$ that is absolutely continuous with regards to $\mu$. 

\begin{claim*}
For each $c\in C$, every formula $\psi(x,bc)$ with $b\in N^k$ and any $B \in \mcB(\Omega)$,
\[
\mbP([\psi(x,bc)] \sqcap B \sqcap A_c)^{p'} = \sum \{ \nu_q(B) : q\in X_c, \psi(x,b) \in q\},
\]
where we are conflating $bc$ with the corresponding tuple of constant functions in $M^{\Omega^2}$.
\end{claim*}
\begin{proof}
We proceed by induction on the Cantor--Bendixson rank of $\psi(x,bc)$, thought of as a subset of $X_c$. We will abridge $d=bc$, $D=B\sqcap A_c$.

If $\psi(x,d)$ has Cantor--Bendixson rank $0$, then it is a finite set of isolated types in $X_c$. Let these types be $q_0,\dots,q_{m-1}$. For each $i<m$ let $\chi_{i}(x)$ be a formula isolating $q_i$. By construction, we must have that $\nu_{q_i}(B) = \mbP([\chi_i(x)]\sqcap D)^{p'}$. Therefore, $\mbP([\psi(x,d)] \sqcap D)^{p'} = \sum_{i<k}\nu_{q_i}(B)$, as required.

Suppose that the statement has been shown for all formulas with Cantor--Bendixson rank less than $\alpha$, and that $\psi(x,d)$ has rank $\alpha$. By decomposing $\psi(x,d)$ into finitely many formulas, we may assume that it has Cantor--Bendixson degree $1$. Let $q\in X_c$ be the unique type of Cantor-Bendixson rank $\alpha$ containing $\psi(x,d)$. By construction, if $r \ni \psi(x,d)$ is a type not equal to $q$, then for some $i\in\mbN$, $\varphi_q^{i}(x) \notin r$. For each $i\in\mbN$, let $\eta_i(x) = \psi(x,d) \wedge \varphi_q^i(x) \wedge \neg \varphi_q^{i+1}(x)$. Note that each $\eta_i(x)$ has Cantor--Bendixson rank strictly less than $\alpha$ and that any type containing $\psi(x,d)$ is either $q$ or contains $\eta_i(x)$ for some $i\in\mbN$. Also note that for sufficiently large $i$, $\psi(x,d)\wedge \varphi^i_q(x) \equiv \varphi^i_q(x)$. For any $i\in\mbN$, we clearly have
  \[
    \mbP([\psi(x,d)]\sqcap D)^{p'} = \mbP([\psi(x,d)\wedge \varphi_q^i(x)]\sqcap D)^{p'}+ \sum_{j<i} \mbP([\eta_j(x)]\sqcap D)^{p'}.
  \]
By the induction hypothesis we have, for each $j\in\mbN$,
  \[
    \mbP([\eta_j(x)]\sqcap D)^{p'} = \sum\{\nu_r(B) : r\in X_c, \eta_j(x) \in r\}.
  \]
Therefore,
  \begin{align*}
    \mbP([\psi(x,d)]\sqcap D)^{p'} &= \lim_{i \to \infty} \mbP([\psi(x,d) \wedge \varphi_q^i(x)]\sqcap D)^{p'} + \sum_{j<i}\mbP([\eta_j(x)]\sqcap D])^{p'} \\
                                            &= \left[\lim_{i \to \infty} \mbP([\psi(x,d)\wedge \varphi_q^i(x)]\sqcap D)^{p'} \right] +  \sum_{j < \omega} \mbP([\eta_j(x)]\sqcap D)^{p'}  \\
                                            &=  \left[\lim_{i \to \infty} \mbP([\varphi_q^i(x)]\sqcap D)^{p'} \right] +  \sum_{j < \omega} \mbP([\eta_j(x)]\sqcap D)^{p'} \\
                                            & = \nu_q(B) + \sum_{j < \omega}\sum\{\nu_r(B) : r\in X_c, \eta_j(x) \in r\} \\
                                            & = \sum\{\nu_r(B) : r\in X_c, \psi(x,d)\in r\},                                              
  \end{align*}
  as required.
\end{proof}

Now we can build a realization of $p$, as follows. For each $c\in C$ and $q\in X_c$, there is a realization $a_q$ of $q$ in $M$, because $(M,N)$ is a beautiful pair. On the other hand, by the Radon--Nikodym theorem, there is a measurable function $\delta_q\colon \Omega \to [0,1]$ such that $\nu_q(B) = \int_B \delta_q\rd\mu$ for every $B\in\mcB(\Omega)$. Note that if we specialize the claim in a true formula $\psi$, we obtain $\mu(B\cap A_c)=\sum_{q\in X_c}\nu_q(B) = \int_B\sum_{q\in X_c}\delta_q\rd\mu$. We deduce that
\[
\sum_{q\in X_c} \delta_q(\omega) = \mu_{\Omega'}((A_c)_\omega)
\]
for almost every $\omega\in\Omega$, where $(A_c)_\omega = \{\omega'\in\Omega : (\omega,\omega')\in A_c\}$ is the slice of $A_c$ at $\omega\in\Omega$.

For each $c\in C$, choose a measurable partition $A_c=\bigsqcup_{q\in X_c}A_q$ such that, for almost every $\omega\in\Omega$,
\[
\mu((A_q)_\omega) = \delta_q(\omega).
\]
We have thus a partition $\Omega^2=\bigsqcup_{c\in C,q\in X_c}A_q$. We let then $f \in M^{\Omega^2}$ be the element defined by
\[
f(\omega,\omega') = a_q\text{ if and only if }(\omega,\omega')\in A_q,
\]
and we claim that $f$ realizes $p$.

By quantifier elimination, it suffices to check that $\mbP[\psi(x,st)]^p = \mu\llbracket\psi(f,st)\rrbracket$ for all $s\in (N^k)^\Omega$. To that end, given such an $s$, let us consider the sets $B_b = \{\omega\in\Omega : s(\omega)=b\}$ for each $b\in N^k$. Then, by the claim, we have that 
\[
  \mbP([\psi(x,st)])\sqcap B_b\sqcap A_c])^{p'} = \sum\{\nu_q(B_b) : q\in X_c, \psi(x,bc) \in q\}
\]
for each $b\in N^k$, and so
\[
  \mbP[\psi(x,st)]^p = \mbP[\psi(x,st)]^{p'} = \sum_{b\in N,c\in C}\sum\{\nu_q(B_b) : q\in X_c, \psi(x,bc) \in q\}.
\]
On the other hand, we have that
\begin{align*}
\mu\llbracket\psi(f,st)\rrbracket = & \sum_{b\in N,c\in C}\sum\left\{ \mu(B_b\cap A_q) : q\in X_c,M\models \psi(a_q,bc) \right\} \\
& \sum_{b\in N,c\in C}\sum\left\{ \int_{B_b}\delta_q\rd\mu : q\in X_c,\psi(x,bc) \in q \right\},
\end{align*}
which by construction is the same quantity. This concludes the proof of Theorem~\ref{thm:beautiful-pair-of-TR}.

\enlargethispage{0.3in}

\bibliographystyle{amsalpha}
\bibliography{biblio}

\providecommand{\bysame}{\leavevmode\hbox to3em{\hrulefill}\thinspace}
\providecommand{\MR}{\relax\ifhmode\unskip\space\fi MR }
\providecommand{\MRhref}[2]{%
  \href{http://www.ams.org/mathscinet-getitem?mr=#1}{#2}
}
\providecommand{\href}[2]{#2}
\begin{thebibliography}{BDNT17}

\bibitem[AZ86]{ahlzie84}
G.~Ahlbrandt and M.~Ziegler, \emph{Quasi-finitely axiomatizable totally
  categorical theories}, Ann. Pure Appl. Logic \textbf{30} (1986), no.~1,
  63--82, Stability in model theory (Trento, 1984). \MR{831437 (87k:03026)}

\bibitem[BBH14]{bbhSFB14}
I.~{Ben Yaacov}, A.~Berenstein, and C.~W. Henson, \emph{Almost indiscernible
  sequences and convergence of canonical bases}, J. Symb. Log. \textbf{79}
  (2014), no.~2, 460--484. \MR{3224976}

\bibitem[BBHU08]{bbhu08}
I.~{Ben Yaacov}, A.~Berenstein, C.~W. Henson, and A.~Usvyatsov, \emph{Model
  theory for metric structures}, Model theory with applications to algebra and
  analysis. {V}ol. 2, London Math. Soc. Lecture Note Ser., vol. 350, Cambridge
  Univ. Press, Cambridge, 2008, pp.~315--427. \MR{2436146 (2009j:03061)}

\bibitem[BDNT17]{msa}
I.~{Ben Yaacov}, M.~Doucha, A.~Nies, and T.~Tsankov, \emph{Metric {Scott}
  analysis}, Advances in Mathematics \textbf{318} (2017), 46--87.

\bibitem[{Ben}08]{benYaacovOnpert}
I.~{Ben Yaacov}, \emph{On perturbations of continuous structures}, Journal of
  Mathematical Logic \textbf{08} (2008), no.~02, 225--249.

\bibitem[{Ben}12]{benOnuniform}
\bysame, \emph{On uniform canonical bases in {$L\sb p$} lattices and other
  metric structures}, J. Log. Anal. \textbf{4} (2012), Paper 12, 30.
  \MR{2955050}

\bibitem[BK09]{benkei}
I.~{Ben Yaacov} and H.~J. Keisler, \emph{Randomizations of models as metric
  structures}, Confluentes Math. \textbf{1} (2009), no.~2, 197--223.
  \MR{2561997}

\bibitem[Bod20]{bodorClass}
B.~Bodor, \emph{Classification of $\omega$-categorical monadically stable
  structures.}, arXiv:2011.08793v1 [math.LO].

\bibitem[BU10]{benusv10}
I.~{Ben Yaacov} and A.~Usvyatsov, \emph{Continuous first order logic and local
  stability}, Trans. Amer. Math. Soc. \textbf{362} (2010), no.~10, 5213--5259.
  \MR{2657678 (2012a:03095)}

\bibitem[CHL85]{cheharlac}
G.~Cherlin, L.~Harrington, and A.~H. Lachlan, \emph{{$\aleph_0$}-categorical,
  {$\aleph_0$}-stable structures}, Ann. Pure Appl. Logic \textbf{28} (1985),
  no.~2, 103--135. \MR{779159}

\bibitem[Han20]{hansonAppIso}
J.~Hanson, \emph{Approximate isomorphism of metric structures},
  arXiv:2011.00588 [math.LO].

\bibitem[Iba17]{ibaRando}
T.~Ibarluc\'{i}a, \emph{Automorphism groups of randomized structures}, J. Symb.
  Log. \textbf{82} (2017), no.~3, 1150--1179. \MR{3694347}

\bibitem[Kec95]{kechrisDST}
A.~S. Kechris, \emph{Classical descriptive set theory}, Graduate Texts in
  Mathematics, vol. 156, Springer-Verlag, New York, 1995. \MR{1321597
  (96e:03057)}

\bibitem[Lac92]{lachlanTree}
A.~H. Lachlan, \emph{{$\aleph_0$}-categorical tree-decomposable structures}, J.
  Symbolic Logic \textbf{57} (1992), no.~2, 501--514. \MR{1169186}

\bibitem[Poi83]{poiPaires}
B.~Poizat, \emph{Paires de structures stables}, J. Symbolic Logic \textbf{48}
  (1983), no.~2, 239--249. \MR{704080}

\end{thebibliography}

\end{document}